\documentclass[12pt]{article}

\usepackage{graphicx}        
\usepackage{multicol}        
\usepackage[bottom]{footmisc}
\usepackage{a4wide,amsmath,amssymb,amsthm,xcolor,booktabs}

\newtheorem{theorem}{Theorem}[section]

\newtheorem{lemma}[theorem]{Lemma}

\begin{document}

\title{Numerical solution of elliptic distributed optimal control
  problems with boundary value tracking
  \footnote{This is an extended version of our paper that is accepted for publication in 
  the DD29 proceedings.}}
\author{Ulrich Langer\footnote{Institute of Numerical Mathematics, 
          JKU,
          Linz,
          Austria, ulanger@numa.uni-linz.ac.at},
Richard L\"oscher\footnote{Institut f\"{u}r Angewandte Mathematik, TU Graz, Austria, loescher@math.tugraz.at}, 
Olaf Steinbach\footnote{Institut f\"{u}r Angewandte Mathematik, TU Graz, Austria, o.steinbach@tugraz.at}, 
Huidong Yang\footnote{Faculty of Mathematics, University of Vienna, Austria, huidong.yang@univie.ac.at}
}

\maketitle

\abstract{
We consider some boundary value tracking optimal control problem 
constrained by a Neumann boundary value problem for some elliptic 
partial differential equation where the control acts as 
right-hand side. This optimal control problem can be reformulated 
as a state-based variational problem that is the starting point 
for the finite element discretizion. In this paper, we only consider a
tensor-product finite element discretizion for which optimal 
discretization error estimates and fast solvers can be derived.
Numerical experiments illustrate the theoretical results quantitatively.
}


\section{Introduction}
\label{Sec1:Introduction}
Let us consider the following boundary tracking optimal control problem (ocp):
Find the state 
$y_\varrho \in H^1(\Omega)$
and the optimal control $u_\varrho \in U$ minimizing the cost functional 
\begin{equation}\label{Eqn:CostFunctional}
   J(y_\varrho,u_\varrho) =
\frac{1}{2} \, \| y_\varrho - \overline{y} \|_{L^2(\Gamma)}^2  + \frac{1}{2} \, \varrho \, \|u_\varrho\|_U^2
\end{equation}
subject to the Neumann boundary value problem (bvp)
\begin{equation}\label{Eqn:NeumannBVP}
  - \Delta y_\varrho + y_\varrho = u_\varrho \quad \mbox{in} \; \Omega, \quad 
  \partial_n y_\varrho = 0 \quad \mbox{on} \; \Gamma, 
\end{equation}
where
$\overline{y} \in L^2(\Gamma)$ denotes a
given target, $\varrho \in {\mathbb{R}}_+$ is a positive regularization or
cost parameter, and $\Omega \subset \mathbb{R}^d$, $d=2,3$, is a bounded
Lipschitz domain with the boundary $\Gamma = \partial \Omega$.
Our work 
was
inspired by the paper \cite{MardalNielsenNordaas2017BIT},
where the $L^2$ regularization corresponding to the choice $U=L^2(\Omega)$
was investigated. 
The boundary tracking is a special case of the partial tracking of a given
target in a subset of $\Omega$ that is also called limited observation;
see, e.g., \cite{MardalSognTakacs2022SINUM}. 
There are many important practical applications, e.g., inverse heat
transfer, photoacoustic tomography, etc., see also \cite{MardalNielsenNordaas2017BIT}.
In this contribution, we consider the energy regularization
corresponding to the choice $U = \widetilde{H}^{-1}(\Omega) := [H^1(\Omega)]^*$
such that the solution of \eqref{Eqn:NeumannBVP} defines an isomorphism
  when considering $y_\varrho \in H^1(\Omega)$. When using $U=L^2(\Omega)$
  this implies $y_\varrho \in H^1_\Delta(\Omega) := \{ y \in H^1(\Omega) :
  \Delta y \in L^2(\Omega)\}$ to ensure an isomorphism, which afterwards
  requires higher order basis functions for a conforming discretization.
Here we use the standard notations for Lebesgue and Sobolev spaces;
see, e.g., \cite{Steinbach2008Book}.
In order to follow the abstract theory presented in
\cite{LangerLoescherSteinbachYang2025MCRF}, we define the state space in such
a way that the state-to-control map is an isomorphism. This allows us to
derive a state-based formulation which is the basis for the numerical solution.
Here we restrict the analysis to a conforming
tensor-product finite element (fe) discretization
that finally leads to a linear system of algebraic equations for which fast
solvers can be constructed. Note that $c$ denotes a universal constant
independent of the regularization parameter $\varrho$, and the
discretization parameter $h$.
%
%

\section{State-based variational reformulation}
\label{Sec3:VF}
The variational formulation of the Neumann 
boundary value problem
\eqref{Eqn:NeumannBVP} reads to find $y_\varrho \in H^1(\Omega)$
such that
\begin{equation}\label{Eqn:NeumannBVPVF}
  \langle \nabla y_\varrho , \nabla y \rangle_{L^2(\Omega)} +
  \langle y_\varrho , y \rangle_{L^2(\Omega)} =
  \langle u_\varrho , y \rangle_{\Omega}
\end{equation}
is satisfied for all $y \in H^1(\Omega)$, where we assume
$u_\varrho \in \widetilde{H}^{-1}(\Omega)$. 
The variational problem has a unique solution due to Lax-Milgram's lemma; see, e.g., \cite{Steinbach2008Book}.
While the Neumann boundary
condition in \eqref{Eqn:NeumannBVP} enters the variational formulation
\eqref{Eqn:NeumannBVPVF} in a natural way, this condition has to be
included in the definition of the state space
\[
  Y := \Big \{ y \in H^1(\Omega) : 
  \langle \partial_n y , \phi \rangle_\Gamma
  = 0 \; \mbox{for all} \; \phi \in H^{1/2}(\Gamma) \Big \}.
\]
When using duality arguments, we then conclude
\[
  \| u_\varrho \|_{Y^*} :=
  \sup\limits_{0 \neq y \in Y}
  \frac{\langle u_\varrho , y \rangle_\Omega}
  {\| y \|_{H^1(\Omega)}} =
  \sup\limits_{0 \neq y \in Y}
  \frac{\langle y_\varrho , y \rangle_{H^1(\Omega)}}
  {\| y \|_{H^1(\Omega)}} =
  \| y_\varrho \|_{H^1(\Omega)},
\]
and instead of \eqref{Eqn:CostFunctional} we can consider the reduced
state-based cost functional
\begin{equation}\label{Eqn:ReducedCostFunctional}
  \widetilde{J}(y_\varrho) =
  \frac{1}{2} \, \| y_\varrho - \overline{y} \|^2_{L^2(\Gamma)} +
  \frac{1}{2} \, \varrho \, \| y_\varrho \|^2_{H^1(\Omega)} ,
\end{equation}
whose minimizer $y_\varrho \in Y$ is the unique solution of the gradient
equation satisfying
\begin{equation}\label{Eqn:GradientEquationVF}
  \langle y_\varrho , y \rangle_{L^2(\Gamma)} + \varrho \,
  \langle y_\varrho , y \rangle_{H^1(\Omega)} =
  \langle \overline{y} , y \rangle_{L^2(\Gamma)} \quad
  \mbox{for all} \; y \in Y.
\end{equation}
Following the abstract theory as given in
\cite[Lemma 2.1]{LangerLoescherSteinbachYang2025MCRF}, we have the following
regularization error estimate.

\begin{lemma}
\label{Lemma:RegularizationEstimates}
  Let $y_\varrho \in Y$ be the unique solution of the variational formulation
  \eqref{Eqn:GradientEquationVF}. For $\overline{y} \in L^2(\Gamma)$
  there hold the estimates
  \begin{equation}\label{Eqn:Regularization error L2}
    \| y_\varrho - \overline{y} \|_{L^2(\Gamma)} \leq
    \| \overline{y} \|_{L^2(\Gamma)}, \quad
    \| y_\varrho \|_{H^1(\Omega)} \leq
    \varrho^{-1/2} \, \| \overline{y} \|_{L^2(\Gamma)} .
  \end{equation}
  If\, $\overline{y} \in H^{1/2}(\Gamma)$ is the Dirichlet trace of\,
  $\overline{y}_e \in Y$, then
  $\| y_\varrho \|_{H^1(\Omega)} \leq \| \overline{y}_e \|_{H^1(\Omega)}$, and
  \begin{equation}\label{Eqn:Regularization H12}
    \| y_\varrho - \overline{y} \|_{L^2(\Gamma)} \leq
    \varrho^{1/2} \, \| \overline{y}_e \|_{H^1(\Omega)}, \quad
    \| y_\varrho - \overline{y}_e \|_{H^1(\Omega)} \leq
    \| \overline{y}_e \|_{H^1(\Omega)}.
  \end{equation}
    If $\overline{y} \in H^1(\Gamma)$ is the Dirichlet trace of
  $\overline{y}_e \in Y \cap H^{3/2+\varepsilon}(\Omega)$ for some
  $\varepsilon > 0$, then
  \begin{equation}\label{Eqn:Regularization error H1}
    \| y_\varrho - \overline{y} \|_{L^2(\Gamma)}
     \leq c \, \varrho \, \| \overline{y} \|_{H^1(\Gamma)}, \quad
    \| y_\varrho - \overline{y}_e \|_{H^1(\Omega)}
     \leq c \, \varrho^{1/2} \, \| \overline{y} \|_{H^1(\Gamma)} .
   \end{equation}
Here, the constant $c$ depends on the mapping properties of the
partial differential operator $I-\Delta$ in certain Sobolev spaces,
and the trace theorem.
\end{lemma}

%
%
\section{Conforming FE discretization on tensor product meshes}
\label{Sec3:FEM}

For a conforming fe discretization of the variational formulation
\eqref{Eqn:GradientEquationVF}, we need to introduce a fe space
$Y_h \subset Y$ of basis functions with zero normal derivatives. 
In this paper, we restrict our considerations to the unit square
$\Omega = (0,1)^d$ which allows us to use appropriate tensor product fe
spaces. Let $\widetilde{S}_h^1(0,1) := \mbox{span} \{ \varphi_i \}_{i=1}^{n-1}$
be the modified fe space of piecewise linear and continuous basis functions
$\varphi_i$ which are defined with respect to a decomposition
$0 = x_0 < x_1 < x_2 < \ldots < x_{n-1} < x_n = 1$ of the unit interval,
with the local mesh sizes $h_i :=x_i-x_{i-1}$, $i=1,\ldots,n$,
and with the global mesh size $h:= \max_i h_i$. While the basis functions
$\varphi_i$ for $i=2,\ldots,n-2$ are the standard piecewise linear and
continuous ones, the basis functions $\varphi_1$ and $\varphi_{n-1}$ are 
$1$ in the intervals $(x_0,x_1)$ and $(x_{n-1},x_n)$, respectively; 
see Fig.~\ref{Fig:Figure Basisfunktionen}.
\begin{figure}[htb]
  \begin{center}
  \unitlength1cm
  \begin{picture}(10,1.5)
    \put(0,0.5){\line(1,0){10}}
    \put(1,0.4){\line(0,1){0.2}}
    \put(2,0.4){\line(0,1){0.2}}
    \put(3,0.4){\line(0,1){0.2}}
    \put(4,0.4){\line(0,1){0.2}}
    \put(6,0.4){\line(0,1){0.2}}
    \put(7,0.4){\line(0,1){0.2}}
    \put(8,0.4){\line(0,1){0.2}}
    \put(9,0.4){\line(0,1){0.2}}
    \put(0.85,0.1){$x_0$}
    \put(1.85,0.1){$x_1$}
    \put(2.85,0.1){$x_2$}
    \put(3.85,0.1){$x_3$}
    \put(5.85,0.1){$x_{n-3}$}
    \put(6.85,0.1){$x_{n-2}$}
    \put(7.85,0.1){$x_{n-1}$}
    \put(8.85,0.1){$x_n$}
    \put(2,0.5){\line(1,1){1}}
    \put(4,0.5){\line(-1,1){1}}
    \put(2.85,1.6){$\varphi_2$}
    \put(6,0.5){\line(1,1){1}}
    \put(8,0.5){\line(-1,1){1}}
    \put(6.85,1.6){$\varphi_{n-2}$}
    \put(1,1.5){\line(1,0){1}}
    \put(3,0.5){\line(-1,1){1}}
    \put(1.85,1.6){$\varphi_1$}
    \put(8,1.5){\line(1,0){1}}
    \put(7,0.5){\line(1,1){1}}
    \put(7.85,1.6){$\varphi_{n-1}$}
  \end{picture}
\end{center}
\caption{Modified piecewise linear basis functions $\varphi_i(x)$, $i=1,\ldots,n-1$.}
\label{Fig:Figure Basisfunktionen}
\end{figure}

\noindent
By construction, we have $\varphi_i'(x) = 0$ for
$x \in \{0,1 \}$, $i=1,\ldots,n-1$.
We now define the conforming fe space
$Y_h = \otimes_{i=1}^d \widetilde{S}_h^1(0,1) =
  \mbox{span} \{ \phi_k \}_{k=1}^m \subset Y $
of piecewise multi-linear continuous basis functions $\phi_k$ with
vanishing Neumann trace $\partial_n \phi_k$ on $\Gamma$. 
We note that $m = (n-1)^d$. The fe discretization of the variational
formulation \eqref{Eqn:GradientEquationVF} leads to the finite element
scheme: Find $y_{\varrho h} \in Y_h$ such that
\begin{equation}\label{Eqn:Gradient equation FEM}
  \langle y_{\varrho h} , y_h \rangle_{L^2(\Gamma)} + \varrho \,
  \langle y_{\varrho h} , y_h \rangle_{H^1(\Omega)} =
  \langle \overline{y} , y_h \rangle_{L^2(\Gamma)} \quad
  \mbox{for all} \; y_h \in Y_h.
\end{equation}
Using standard arguments, we immediately arrive at the Cea-type estimate
\begin{eqnarray}\label{Eqn:Cea}
  && \| y_\varrho - y_{\varrho h} \|_{L^2(\Gamma)}^2 +
     \varrho \, \| y_\varrho - y_{\varrho h} \|^2_{H^1(\Omega)} \\
  && \hspace*{2cm} \leq
     \| y_\varrho - y_h \|_{L^2(\Gamma)}^2 +
     \varrho \, \| y_\varrho - y_h \|^2_{H^1(\Omega)} \quad \mbox{for all}
     \; y_h \in Y_h . \nonumber
\end{eqnarray}

\begin{lemma}
  Let $y_{\varrho h} \in Y_h$ be the unique solution of
  \eqref{Eqn:Gradient equation FEM}. Then,
  for $\overline{y} \in L^2(\Gamma)$, there holds the error estimate
  \begin{equation}\label{Eqn:Energie Error L2}
    \| y_{\varrho h} - \overline{y} \|_{L^2(\Gamma)} \leq (1+\sqrt{2}) \,
    \| \overline{y} \|_{L^2(\Gamma)}.
  \end{equation}
  If $\overline{y} \in H^{1/2}(\Gamma)$ is the Dirichlet trace of
    $\overline{y}_e \in Y$, and choosing $\varrho=h$, we have
  \begin{equation}\label{Eqn:Energie Error H12}
    \| y_{\varrho h} - \overline{y} \|_{L^2(\Gamma)} \leq c \, h^{1/2} \,
    \| \overline{y} \|_{H^{1/2}(\Gamma)} .
  \end{equation}
  For $\overline{y} = \overline{y}_{e|\Gamma} \in H^1(\Gamma)$,
    $\overline{y}_e \in H^{3/2+\varepsilon}(\Omega)$ with $\varepsilon > 0$, 
    and choosing $\varrho=h$, we have
  \begin{equation}\label{Eqn:Energie Error H1}
    \| y_{\varrho h} - \overline{y} \|_{L^2(\Gamma)} \leq c \, h \,
    \| \overline{y} \|_{H^1(\Gamma)} .
  \end{equation}
\end{lemma}

\begin{proof}
  For $\overline{y} \in L^2(\Gamma)$, we consider \eqref{Eqn:Cea}
  for $y_h=0$, and we use 
  \eqref{Eqn:GradientEquationVF} and 
  \eqref{Eqn:Regularization error L2} to obtain
  \[
    \| y_\varrho - y_{\varrho h} \|_{L^2(\Gamma)}^2 
    \leq
    \| y_\varrho \|_{L^2(\Gamma)}^2 +
    \varrho \, \| y_\varrho \|^2_{H^1(\Omega)} \leq 2 \,
    \| \overline{y} \|_{L^2(\Gamma)}^2 .
  \]
  With the triangle inequality and again using
  \eqref{Eqn:Regularization error L2} we therefore conclude
  \[
    \| y_{\varrho h} - \overline{y} \|_{L^2(\Gamma)} \leq
    \| y_{\varrho h} - y_\varrho \|_{L^2(\Gamma)} +
    \| y_\varrho - \overline{y} \|_{L^2(\Gamma)} \leq
    (1+\sqrt{2}) \, \| \overline{y} \|_{L^2(\Gamma)} \, .
  \]
For $\overline{y} = \overline{y}_{e|\Gamma} \in H^{1/2}(\Gamma)$
we use the triangle inequality twice and Cea's
estimate \eqref{Eqn:Cea} for arbitrary $y_h \in Y_h$ to write
\begin{eqnarray}\nonumber
    \| y_{\varrho h} - \overline{y} \|_{L^2(\Gamma)}^2
    & \leq & 2 \, \| y_\varrho - \overline{y} \|_{L^2(\Gamma)}^2 +
             2 \, \| y_\varrho - y_{\varrho h} \|^2_{L^2(\Gamma)} \\ \nonumber
    & \leq & 2 \, \| y_\varrho - \overline{y} \|_{L^2(\Gamma)}^2 +
             2 \, \| y_\varrho - y_h \|_{L^2(\Gamma)}^2 +
             2 \, \varrho \, \| y_\varrho - y_h \|^2_{H^1(\Omega)} \\
    \label{Eqn:Proof Step 3}
    & \leq & 6 \, \| y_\varrho - \overline{y} \|_{L^2(\Gamma)}^2 +
             4 \, \| \overline{y} - y_h \|_{L^2(\Gamma)}^2 \\
    && \hspace*{2cm} \nonumber
       + \, 4 \, \varrho \, \| y_\varrho - \overline{y}_e \|_{H^1(\Omega)}^2 +
       4 \, \varrho \, \| \overline{y}_e - y_h \|^2_{H^1(\Omega)} .
\end{eqnarray}
  In particular, for $y_h = P_h \overline{y}_e \in Y_h$
  being the $L^2$ projection of $\overline{y}_e$, we have the standard fe
  error estimates; see, e.g., \cite{Steinbach2008Book},
  \begin{equation}\label{Langer:Projection error}
    \| \overline{y}_e - P_h \overline{y}_e \|_{H^1(\Omega)} \leq c \,
    \| \overline{y}_e \|_{H^1(\Omega)}, \quad
    \| \overline{y}_e - P_h \overline{y}_e \|_{L^2(\Omega)} \leq c \,
    h \, \| \overline{y}_e \|_{H^1(\Omega)}.
  \end{equation}
  When using \cite[Theorem 3.6]{BehrndtGesztesyMitrea2025} and a space
  interpolation argument, we also have
  \[
    \| \overline{y}_e - P_h \overline{y}_e \|_{L^2(\Gamma)} \leq
    c \, \| \overline{y}_e - P_h \overline{y}_e \|_{H^{1/2}(\Omega)} \leq
    c \, h^{1/2} \, \| \overline{y}_e \|_{H^1(\Omega)} .
  \]
  Hence, using \eqref{Eqn:Regularization H12} and
  $\| \overline{y}_e \|_{H^1(\Omega)} \leq \| \overline{y} \|_{H^{1/2}(\Gamma)}$,
  this gives
  \[
    \| y_{\varrho h} - \overline{y} \|_{L^2(\Gamma)}^2
    \, \leq \, c \, (\varrho + h) \, \| \overline{y} \|_{H^{1/2}(\Gamma)}^2 ,
  \]
  and \eqref{Eqn:Energie Error H12} follows when choosing $\varrho=h$.
  
  Finally, we consider the case 
    $\overline{y} =
    \overline{y}_{e|\Gamma} \in H^1(\Gamma)$ with
    $\overline{y}_e \in H^{3/2+\varepsilon}(\Omega)$ for some
    $\varepsilon > 0$.
    In this case, we have
  \[
    \| \overline{y}_e - P_h \overline{y}_e \|_{H^1(\Omega)} \leq c \, h^{1/2} \,
    \| \overline{y}_e \|_{H^{3/2}(\Omega)}, \;
    \| \overline{y}_e - P_h \overline{y}_e \|_{L^2(\Omega)} \leq c \,
    h^{3/2} \, \| \overline{y}_e \|_{H^{3/2}(\Omega)},
  \]
  and \;
  $
    \| \overline{y} - P_h \overline{y}_e \|_{L^2(\Gamma)} \leq c \, h \,
    \| \overline{y}_e \|_{H^{3/2}(\Omega)} .
  $
  Together with \eqref{Eqn:Regularization error H1} we therefore conclude
  $
    \| y_{\varrho h} - \overline{y} \|_{L^2(\Gamma)}^2
    \, \leq \, c \, ( \varrho^2 + \varrho h + h^2) \,
    \| \overline{y}_e \|_{H^{3/2}(\Omega)}^2,
  $
  and, for $\varrho = h$, this gives \eqref{Eqn:Energie Error H1}.
\end{proof}

\noindent
While the regularization error estimates as given in
Lemma \ref{Lemma:RegularizationEstimates} are optimal in $\varrho$ for
$\overline{y} \in H^1(\Gamma)$, i.e.,
$\overline{y}_e \in H^{3/2}(\Omega)$, we can expect higher order
convergence for the fe approximation $P_h \overline{y}_e$ when
$\overline{y}_e$ is more regular. In particular for
$\overline{y}_e \in H^2(\Omega)$ we have
\[
  \| \overline{y}_e - P_h \overline{y}_e \|_{H^1(\Omega)} \leq c \, h \,
  \| \overline{y}_e \|_{H^2(\Omega)}, \;
  \| \overline{y}_e - P_h \overline{y}_e \|_{L^2(\Omega)} \leq c \,
  h^2 \, \| \overline{y}_e \|_{H^2(\Omega)} .
\]
Note that $\overline{y}_e \in H^2(\Omega)$ implies some
additional compatibility conditions on $\overline{y}$,
when $\Gamma = \partial \Omega$
is piecewise smooth. In particular, the tangential derivatives
of $\overline{y}$ have to be in $H^{1/2}(\Gamma)$, see, e.g., the
discussion in \cite{Jakovlev:1961,KhoromskijSchmidt:1999} in the
two-dimensional case. Then we can use \eqref{Eqn:Proof Step 3}
to conclude
\[
  \| y_{\varrho h} - \overline{y} \|_{L^2(\Gamma)}^2
  \, \leq \,
  c_1 \, \varrho^2 \, \| \overline{y} \|_{H^1(\Gamma)}^2 +
  c_2 \, (h^3 + \varrho \, h^2) \, \| \overline{y}_e \|^2_{H^2(\Omega)} ,
\]
and for $\varrho \leq h^{3/2}$ we obtain
\begin{equation}\label{Eqn:Error H32}
 \| y_{\varrho h} - \overline{y} \|_{L^2(\Gamma)}
  \, \leq \,
  c \, h^{3/2} \, \Big[ \| \overline{y}_e \|^2_{H^2(\Omega)} +
  \| \overline{y} \|_{H^{3/2}(\Gamma)}^2 \Big]^{1/2} .
\end{equation}
However, for $\overline{y} \in H^2(\Gamma)$, i.e.,
$\overline{y}_e \in H^{5/2}(\Omega)$, and using the best approximation
estimate for the boundary term, this gives
\[
  \| y_{\varrho h} - \overline{y} \|_{L^2(\Gamma)}^2 \, \leq \,
  c_1 \, h^4 \, \| \overline{y} \|_{H^2(\Gamma)}^2
  + c_2 \, \varrho^2 \, \| \overline{y} \|_{H^1(\Gamma)}^2 +
  c_3 \, \varrho \, h^2 \, \| \overline{y}_e \|^2_{H^2(\Omega)},
\]
and for $\varrho = h^2$ we finally obtain
\begin{equation}\label{Eqn:Error H2}
  \| y_{\varrho h} - \overline{y} \|_{L^2(\Gamma)} \, \leq \,
  c \, h^2 \Big[ \| \overline{y} \|_{H^2(\Gamma)}^2
  + \| \overline{y} \|_{H^1(\Gamma)}^2 +
  \| \overline{y}_e \|^2_{H^2(\Omega)} \Big]^{1/2} .
\end{equation}

%
%
\section{Fast solvers}
\label{Sec4:Solvers} 
Once the basis is chosen, the finite element scheme
\eqref{Eqn:Gradient equation FEM}
is equivalent to a linear system of finite element equations
that can be written in the form
\begin{equation}\label{Eqn:LinearSystem}
  [\widetilde{\mathbf{M}}_h + \varrho(\widetilde{\mathbf{K}}_h +
  \mathring{\mathbf{K}}_h)] {\mathbf y }_h = \mathbf{\overline{y}}_h,
\end{equation}
where the matrices $\widetilde{\mathbf{M}}_h$, $\widetilde{\mathbf{K}}_h$,
and $\mathring{\mathbf{K}}_h$ have the respective block representations
\begin{equation*}
\widetilde{\mathbf{M}}_h =
\left(
    \begin{array}{cc}
      {\mathbf 0} & {\mathbf 0} \\
      {\mathbf 0} & {\mathbf M}_{BB}
    \end{array}
\right),\quad
\widetilde{\mathbf{K}}_h =
\left(
    \begin{array}{cc}
      {\mathbf 0} & {\mathbf 0} \\
      {\mathbf 0} & \widetilde{\mathbf K}_{BB}
    \end{array}
\right), \quad
\mbox{and} \quad
\mathring{\mathbf{K}}_h =
\left(
    \begin{array}{cc}
      \mathring{\mathbf K}_{II} & \mathring{\mathbf K}_{IB} \\
      \mathring{\mathbf K}_{BI} & \mathring{\mathbf K}_{BB}
    \end{array}
\right),
\end{equation*}
when we split the unknowns (dofs) 
$\mathbf{y}_h = (\mathbf{y}_I^\top,\mathbf{y}_B^\top)^\top \in
\mathbb{R}^{(n-1)^d}$
into strict interior unknowns $\mathbf{y}_I  \in \mathbb{R}^{(n-3)^d}$ 
and near-boundary unknowns $\mathbf{y}_B  \in \mathbb{R}^{(n-1)^d-(n-3)^d}$.
The matrices ${\mathbf M}_{BB}$, $\widetilde{\mathbf K}_{BB}$ and
$\mathring{\mathbf{K}}_h$ are defined by the identities
$({\mathbf M}_{BB}\mathbf{y}_B,\mathbf{v}_B) = \langle y_h , v_h
\rangle_{L^2(\Gamma)}$,
$(\widetilde{\mathbf K}_{BB}\mathbf{y}_B,\mathbf{v}_B) =
\langle y_h , v_h \rangle_{H^1(\widetilde{\Omega}_h)}$, 
$(\mathring{\mathbf{K}}_h\mathbf{y}_h,\mathbf{v}_{h}) =
\langle y_h , v_h \rangle_{H^1(\mathring{\Omega}_h)}$
for all $\mathbf{y}_h = (\mathbf{y}_I^\top,\mathbf{y}_B^\top)^\top
\leftrightarrow y_h,v_h \in Y_h$ (fe  isomorphism),
where
$\mathring{\Omega}_h = \Omega \setminus \overline{\widetilde{\Omega}}_h
= (h,1-h)^d$ and $\widetilde{\Omega}_h =
\Omega \setminus \overline{\mathring{\Omega}}_h $,
whereas $\mathbf{\overline{y}}_h =
(\mathbf{0}_I^\top,\mathbf{\overline{y}}_B^\top)^\top \in \mathbb{R}^m$
is given by $(\mathbf{\overline{y}}_B, \mathbf{y}_B) =
\langle  \overline{y}_h , y_h \rangle_{L^2(\Gamma)}$
for all $\mathbf{y}_h = (\mathbf{y}_I^\top,\mathbf{y}_B^\top)^\top
\leftrightarrow y_h,v_h \in Y_h$.

Eliminating $\mathbf{y}_I = - \mathring{\mathbf K}_{II}^{-1} \mathring{\mathbf K}_{IB}\mathbf{y}_B$
from the linear system \eqref{Eqn:LinearSystem}, we arrive at the boundary Schur complement (SC)
system
\begin{equation}\label{Eqn:SchurComplemetSystem}
{\mathbf S}_{BB} \mathbf{y}_B = \mathbf{\overline{y}}_B
\end{equation}
with ${\mathbf S}_{BB} = {\mathbf M}_{BB} + \varrho (\widetilde{\mathbf K}_{BB} +  \mathring{\mathbf S}_{BB})
                       = {\mathbf M}_{BB} + \varrho (\widetilde{\mathbf K}_{BB} +  
      (\mathring{\mathbf K}_{BB} - \mathring{\mathbf K}_{BI} \mathring{\mathbf K}_{II}^{-1} \mathring{\mathbf K}_{IB}))$.
The Schur complement system \eqref{Eqn:SchurComplemetSystem} can efficiently be solved by means of 
the Conjugate Gradient (CG) method without any preconditioning 
since, for $\varrho \le h$, the Schur complement ${\mathbf S}_{BB}$ is spectrally equivalent to 
the boundary mass matrix $\mathbf{M}_{BB}$,
and in turn $\mathbf{M}_{BB}$ is spectrally equivalent to 
the lumped boundary mass matrix $\mbox{lump}(\mathbf{M}_{BB})$
and to $h^{d-1}{\mathbf I}_{BB}$.
Indeed, it is easy to see that
\begin{equation}\label{Eqn:SpectralEquivalenceInequalities}
{\mathbf M}_{BB} \le  {\mathbf S}_{BB} = {\mathbf M}_{BB} + \varrho (\widetilde{\mathbf K}_{BB} +  \mathring{\mathbf S}_{BB})
                \le (1 + \widetilde{c} \varrho h^{-1} +  \mathring{c} \varrho h^{-1}) {\mathbf M}_{BB} 
\end{equation}
with $h$ and $\varrho$ independent positive constants $\widetilde{c}$ and $\mathring{c}$
arising from the estimates $\lambda_{max}(\mathbf{M}_{BB}^{-1}\widetilde{\mathbf K}_{BB}) \le \widetilde{c} h^{-1}$
and $\lambda_{max}(\mathbf{M}_{BB}^{-1} \mathring{\mathbf S}_{BB}) \le \mathring{c} h^{-1}$
of the maximal eigenvalues of $\mathbf{M}_{BB}^{-1}\widetilde{\mathbf K}_{BB}$ and     
$\mathbf{M}_{BB}^{-1} \mathring{\mathbf S}_{BB}$, respectively. 
The choice $\varrho \le h$ delivers the desired result.
It is recommended to use $\mbox{lump}(\mathbf{M}_{BB})$ as diagonal preconditioner 
in the Preconditioned Conjugate Gradient (PCG) method since it provides the right scaling.
The numerical results presented in Section~\ref{Sec5:NumericalResults} show
that 
the system \eqref{Eqn:LinearSystem} can also efficiently  be solved by means of 
PCG with a simple Algebraic MultiGrid (AMG) preconditioner.

%
%
\section{Numerical results}
\label{Sec5:NumericalResults}
We first consider the target
\begin{equation}\label{Eqn:TargetHomogeneousNeumann}
  \overline{y} = \overline{y}(x) :=
  \cos(\pi x_1)\cos(\pi x_2)\cos(\pi x_3), \; x=(x_1,x_2,x_3) \in \Gamma,
\end{equation}
on the boundary $\Gamma=\partial \Omega$ of the domain $\Omega=(0,1)^3$.
We mention that $\overline{y}$ is the trace of a smooth function with
vanishing  normal derivative on the boundary $\Gamma$, i.e., we have
$\overline{y} \in H^2(\Gamma)$, and the error estimate
\eqref{Eqn:Error H2} applies when choosing $\varrho=h^2$.
We use a tensor product mesh as described in Section~\ref{Sec3:FEM}.
The initial mesh contains $5$ vertices in each direction,
and $125$ in total with mesh size $h=0.25$. 
We note that  we have only $3$ dofs in each direction,
and $27$ in total for the initial level.
Table~\ref{Tab:Examaple_Target_Homogeneous_Neuammn_tensor} provides the
numerical results starting from level $\ell = 1$ with $27$ dofs and
running to the finest discretization level $\ell = 7$
obtained by $6$ uniform refinements of the initial mesh.
The fourth column displays the $L^2$ error $\|y_\ell-\overline{y}\|_{L^2(\Gamma)}$
between the computed fe solution $y_\ell = y_{\varrho_\ell h_\ell}$ and the target $\overline{y}$ on the boundary. 
As expected, we observe second order of convergence;
cf. experimental order of convergence (eoc) given in the fifth column.
We first solve the original system \eqref{Eqn:LinearSystem} by means of 
AMG preconditioned CG iterations (\#AMG-PCG its), and observe that not more 
than $4$ iterations are needed in order to reach a relative residual
error of $10^{-9}$.
We further test the CG and lumped mass preconditioned CG 
solvers for the Schur complement equation \eqref{Eqn:SchurComplemetSystem}
until the relative residual error reaches $10^{-9}$. The number of Schur
complement CG (\#SCG its) and lumped mass preconditioned CG (\#SPCG its)
iterations are displayed in the last two columns of
Table \ref{Tab:Examaple_Target_Homogeneous_Neuammn_tensor}.
As expected from the theoretical results given in Section~\ref{Sec4:Solvers},
we see level-independent iteration numbers in both cases. Moreover, 
the lumped-mass preconditioner further reduces the number of iteration by
the scaling effect. We note that the action of $\mathring{\mathbf K}_{II}^{-1}$
to a vector within the multiplication of the Schur complement
$\mathbf{S}_{BB}$ by some vector (iterate) 
is realized by an AMG preconditioned CG method until the relative residual
error is reduced by a factor $10^{10}$. The latter accuracy of this inner PCG
iteration can be adapted (reduced~!) to the outer CG/PCG iteration
following the results from \cite{SimonciniSzyld2003SISC}.

\begin{table}[ht]
  \centering
  \begin{tabular}{|c|l|c|c|c|c|c|c|}
    \hline
    $\ell$ & \#dofs & $h$  & error & eoc &\#AMG-PCG its &\#SCG its &\#PSCG its\\
    \hline
    $1$ & $27$ & $2^{-2}$ & $1.669$e$-1$ &$-$ & $2$ & $1$ & $1$ \\
    $2$ & $343$ & $2^{-3}$ & $5.215$e$-2$ &$1.68$ & $3$& $6$ & $6$\\
    $3$ & $3,375$  & $2^{-4}$ & $1.347$e$-2$ &$1.95$ & $4$&$17$& $9$ \\
    $4$ & $29,791$ & $2^{-5}$ & $3.226$e$-3$ & $2.06$ & $4$&$24$& $9$ \\
    $5$ & $250,047$ & $2^{-6}$ & $7.707$e$-4$ & $2.07$ & $4$&$28$& $8$ \\
    $6$ & $2,048,383$ & $2^{-7}$ & $1.872$e$-4$ & $2.04$ & $4$&$29$& $6$ \\
    $7$ & $16,581,375$ & $2^{-8}$ & $4.605$e$-5$& $2.02$ & $4$&$29$& $4$ \\
    \hline
  \end{tabular}
  \caption{Target \eqref{Eqn:TargetHomogeneousNeumann}: error = $\|y_\ell-\overline{y}\|_{L^2(\Gamma)}$, 
  eoc = $\log_2 (\|y_\ell-\overline{y}\|_{L^2(\Gamma)} / \|y_{\ell+1}-\overline{y}\|_{L^2(\Gamma)})$,
   number of AMG CG iterations (\#AMG-PCG its) for the original system \eqref{Eqn:LinearSystem}, and
   number of CG  (\#SCG its) and  lumped-mass preconditioned CG (\#PSCG its) iterations 
   for the 
   SC
   system \eqref{Eqn:SchurComplemetSystem}, $\varrho = h^2$.}
  \label{Tab:Examaple_Target_Homogeneous_Neuammn_tensor}
\end{table}

\noindent
The second target
\begin{equation}\label{Eqn:TargetInhomogeneousNeumann}
  \overline{y}(x) := x_1^2-0.5x_2^2-0.5x_3^2, \;
  x=(x_1,x_2,x_3) \in \Gamma,
\end{equation}
is a trace of a smooth function which does not fulfill the homogeneous
Neumann boundary conditions on the boundary $\Gamma$ of $\Omega = (0,1)^3$,
i.e., $\overline{y} \in H^{3/2-\varepsilon}(\Gamma)$, and
$\overline{y}_e \in H^{2-\varepsilon}(\Omega)$, $\varepsilon >0$.
Similar as in \eqref{Eqn:Error H32} we therefore expect a reduced eoc
of about 1.5. We perform the same tests for the target
\eqref{Eqn:TargetInhomogeneousNeumann} as for the previous example for
the target \eqref{Eqn:TargetHomogeneousNeumann}. The results are given in
Table~\ref{Tab:Examaple_Target_NonHomogeneous_Neuammn_tensor}, where we have
used $\varrho = h^{3/2}$ as prescribed by the theory. 
The iteration numbers are again independent of $h$, and 
show the same behavior as in the case of the first example. 

\begin{table}[ht]
  \centering
  \begin{tabular}{|c|l|c|c|c|c|c|c|}
    \hline
    $\ell$ & \#Dofs & $h$  & error & eoc &\#AMG-PCG its &\#SCG its&\#PSCG its\\
    \hline
    $1$ & $27$ & $2^{-2}$ & $4.869$e$-1$ &$-$ & $2$ & $5$ & $5$\\
    $2$ & $343$ & $2^{-3}$ & $2.131$e$-1$ &$1.19$ & $3$ & $13$& $10$ \\
    $3$ & $3,375$  & $2^{-4}$ & $8.177$e$-2$ &$1.38$ & $5$ &$14$& $8$ \\
    $4$ & $29,791$ & $2^{-5}$ & $2.946$e$-2$ & $1.47$ & $4$&$15$& $8$ \\
    $5$ & $250,047$ & $2^{-6}$ & $1.036$e$-2$ & $1.51$ & $4$&$16$& $7$\\
    $6$ & $2,048,383$ & $2^{-7}$ & $3.619$e$-3$ & $1.52$ & $4$&$18$& $7$ \\
    $7$ & $16,581,375$ & $2^{-8}$ & $1.265$e$-3$& $1.52$ & $4$&$19$& $7$ \\
    \hline
  \end{tabular}
   \caption{Target \eqref{Eqn:TargetInhomogeneousNeumann}: Same agenda as Table~\ref{Tab:Examaple_Target_Homogeneous_Neuammn_tensor}, but now $\varrho = h^{3/2}$.}
  \label{Tab:Examaple_Target_NonHomogeneous_Neuammn_tensor}
\end{table}

Finally, we consider a discontinuous target $\overline{y}$ that is equal to 1 inside 
the region $[0.25,  0.75]\times[0.25, 0.75]$ and $0$ outside 
on each face $\Gamma_i \subset \Gamma = \partial \Omega, i=1,\dots,6$   the unit cube $\Omega = (0,1)^3$.
We perform the same tests  as for the previous examples.
The results are given in Table~\ref{Tab:Examaple_Target_Discontinuous_tensor}, 
where we have used $\varrho = h$ as prescribed by the theory. 
We observe the expected convergence rate eoc = 0.50. 
Furthermore, the iteration numbers are again independent of $h$, and 
show the same behavior as in the case of the preceding  two examples.
Figure~\ref{Fig:Examaple_Target_Discontinuous_tensor} shows the computed state $y_\ell$ on the 
surface $\Gamma$ of the cube $\Omega$ (left) and on three cuts through the cube.
\begin{table}[ht]
  \centering
  \begin{tabular}{|c|l|c|c|c|c|c|c|}
    \hline
    $\ell$ & \#Dofs & $h$  & error & eoc &\#AMG-PCG its &\#SCG its&\#PSCG its\\
    \hline
    $1$ & $27$ & $2^{-2}$ & $9.817$e$-2$ &$-$ & $3$ & $3$ & $3$\\
    $2$ & $343$ & $2^{-3}$ & $6.535$e$-1$ &$-$ & $4$ & $10$ & $10$ \\
    $3$ & $3,375$  & $2^{-4}$ & $4.768$e$-1$ & $0.45$ & $5$ & $17$ & $17$ \\
    $4$ & $29,791$ & $2^{-5}$ & $3.530$e$-1$ & $0.43$ & $4$& $22$& $19$ \\
    $5$ & $250,047$ & $2^{-6}$ & $2.540$e$-1$ & $0.47$ & $4$ & $23$ & $21$\\
    $6$ & $2,048,383$ & $2^{-7}$ & $1.809$e$-1$ & $0.49$ & $4$ & $23$ & $21$ \\
    $7$ & $16,581,375$ & $2^{-8}$ & $1.283$e$-1$& $0.50$ & $5$ & $23$ & $21$ \\
    \hline
  \end{tabular}
   \caption{Discontinuous target with 1 inside the region of $[0.25\quad 0.75]\times[0.25\quad  0.75]$ and 0 outside 
   this region on each face of the unique cube:  Same agenda as Table~\ref{Tab:Examaple_Target_Homogeneous_Neuammn_tensor}, but now $\varrho = h$.}
  \label{Tab:Examaple_Target_Discontinuous_tensor}
\end{table}

\begin{figure}
  \centering
  \includegraphics[scale=0.185]{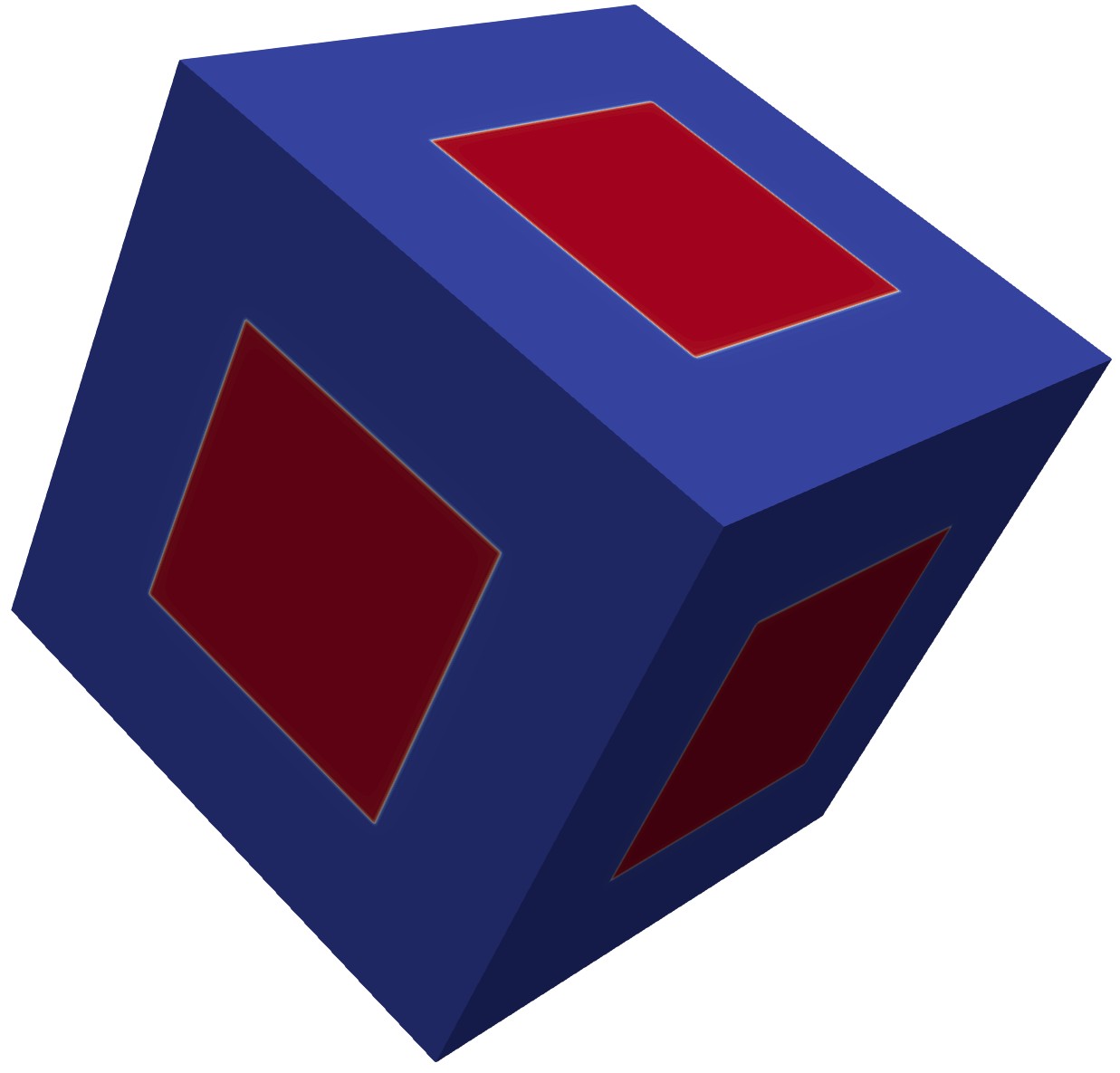}
  \includegraphics[scale=0.185]{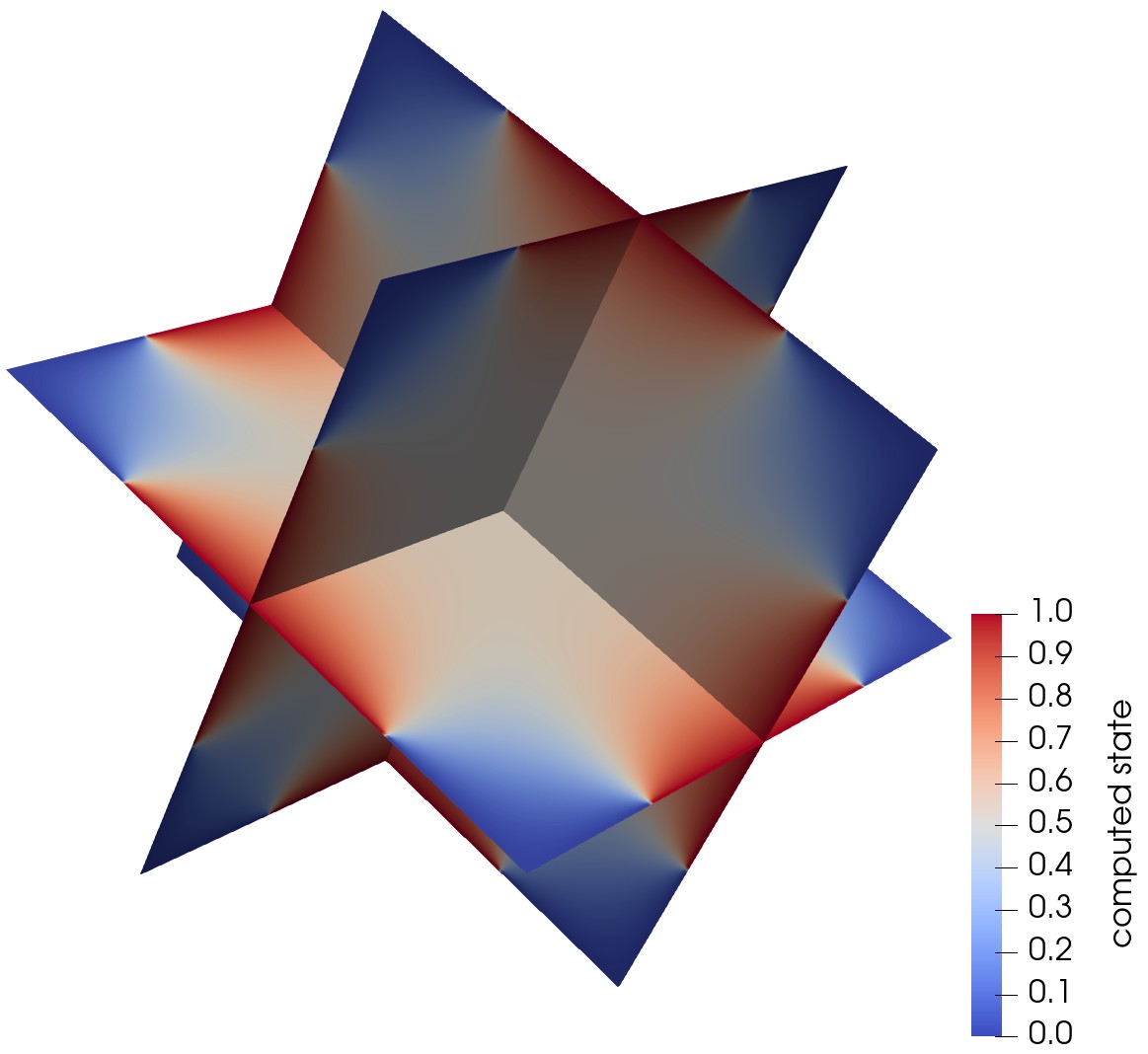}
  \caption{Computed state for the discontinuous target at level $7$.}
  \label{Fig:Examaple_Target_Discontinuous_tensor}
\end{figure}

%
%
\section{Conclusions and Outlook} 
\label{Sec6:Conclusions}
We have investigated a boundary value tracking, distributed ocp
for some elliptic bvp with homogeneous Neumann boundary conditions.
This ocp can be reduced to a state-based variational problem 
for functions from $H^1(\Omega)$ with vanishing normal derivatives in
$H^{-1/2}(\Gamma)$. We have constructed a conforming fe discretization for 
the special case of tensor-product meshes, 
and we have derived discretization error estimates and fast solvers.
The numerical experiments illustrate the theoretical results quantitatively.
In general, one has to include the homogeneous Neumann boundary conditions 
by means of Lagrange multipliers, which is a topic 
of future research. 
%
%
\section*{Acknowledgement}
We would like to thank RICAM for the computing resource support.
The financial support for the fourth author by the
Austrian Federal Ministry for Digital and Economic Affairs, the National
Foundation for Research, Technology and Development, and the Christian
Doppler Research Association is gratefully acknowledged.
%
%

%
%
\bibliographystyle{abbrv}
\bibliography{LangerLoescherSteinbachYang}
%
%
\newpage
\section*{Appendix}
In this appendix we provide the proofs of all results as used in this paper,
where $\widetilde{S}_h^1(0,1)$ is the modified finite element space of
piecewise linear and continuous basis functions $\varphi_i(x)$ for
$x \in (0,1)$ with zero derivatives $\varphi_i'(0)=\varphi_i'(1)=0$.

\begin{lemma}
  Let $y \in C([0,1])$ be a given continuous function with
  $y'(0)=y'(1)=0$, and satisfying $y'' \in L^2(0,1)$.
  For the piecewise linear interpolation
  \[
    I_h y(x) = \sum\limits_{i=1}^{n-1} y(x_i) \varphi_i(x).
  \]
  there hold the error estimates
  \begin{equation}\label{Appendix Interpolation Error L2 H2}
    \| y - I_h y \|_{L^2(0,1)} \leq \frac{1}{2} \, h^2 \,
    \| y'' \|_{L^2(0,1)},
  \end{equation}
  \begin{equation}\label{Appendix Interpolation Error L2 H1}
    \| y - I_hy \|_{L^2(0,1)} \leq \sqrt{2} \, h \, \| y' \|_{L^2(0,1)},
  \end{equation}
  and
  \begin{equation}\label{Appendix Interpolation Error H1 H2}
    \| y' - (I_h y)' \|_{L^2(0,1)} \leq \frac{1}{\sqrt{2}} \, h \,
    \| y'' \|_{L^2(0,1)} .
  \end{equation}
  Moreover, the interpolation operator is stable in $H^1(0,1)$, i.e.,
  \begin{equation}\label{Appendix Interpolation Stability}
    \| (I_hy)' \|_{L^2(0,1)} \leq \| y' \|_{L^2(0,1)} \quad
    \mbox{for all} \; y \in H^1(0,1) .
  \end{equation}
\end{lemma}

\begin{proof}
  For $ i=2,\ldots,n-1$ we have the standard interpolation error estimates
  \[
    \int_{x_{i-1}}^{x_i} [y(x)-I_hy(x)]^2 \, dx \leq
    \frac{1}{24} \, h^4 \, \int_{x_{i-1}}^{x_i} [y''(x)]^2 \, dx \, ,
  \]
  and
  \[
    \int_{x_{i-1}}^{x_i} [y'(x)-(I_hy)'(x)]^2 \, dx \leq
    \frac{1}{3} \, h^2 \, \int_{x_{i-1}}^{x_i} [y''(x)]^2 \, dx \, ,
  \]
  For the first interval $(x_0,x_1)$ we have $I_hy(x_1)=y(x_1)$, and hence
  we can write for $x \in (x_0,x_1)$
  \[
    y(x) - I_hy(x) =
    [I_hy(x_1)-I_hy(x)] - [y(x_1)-y(x)] = \int_x^{x_1} [(I_hy)'(s)-y'(s)] \, ds,
  \]
  and with the Cauchy--Schwarz inequality we obtain
  \begin{eqnarray*}
    [y(x) - I_hy(x)]^2
    & = & \left( \int_x^{x_1} [(I_hy)'(s)-y'(s)] \, ds \right)^2 \\
    & \leq & \int_x^{x_1} 1^2 \, ds
             \int_x^{x_1} [(I_hy)'(s)-y'(s)]^2 \, ds \\
    & \leq & (x_1-x)
             \int_{x_0}^{x_1} [(I_hy)'(s)-y'(s)]^2 \, ds,
  \end{eqnarray*}
  i.e.,
  \[
    \int_{x_0}^{x_1} [y(x) - I_hy(x)]^2 \, dx
    \leq
    \frac{1}{2} \, h^2 \, \int_{x_0}^{x_1} [(I_hy)'(x)-y'(x)]^2 \, dx .
  \]
  Due to $y'(x_0) = (I_hy)'(x_0)=0$ we further have
  \begin{eqnarray*}
    (I_hy)'(x)-y'(x)
    & = & [(I_hy)'(x)-(I_hy)'(x_0)] - [y'(x)-y'(x_0)] \\
    & = & \int_{x_0}^x [(I_hy)''(s)-y''(s)] \, ds \, = \,
          \int_{x_0}^x [-y''(s)] \, ds,
  \end{eqnarray*}
  and hence,
  \begin{eqnarray*}
    [(I_hy)'(x)-y'(x)]^2
    & = & \left( \int_{x_0}^x [-y''(s)] \, ds \right)^2 \leq
          \int_{x_0}^x 1^2 \, ds \int_{x_0}^x [y''(s)]^2 \, ds \\
    & \leq & (x-x_0) \int_{x_0}^{x_1} [y''(s)]^2 \, ds,
  \end{eqnarray*}
  i.e.,
  \[
    \int_{x_0}^{x_1} [(I_hy)'(x)-y'(x)]^2 dx 
    \leq \frac{1}{2} \, h^2 \, \int_{x_0}^{x_1} [y''(x)]^2 \, dx
  \]
  follows, and we also conclude
  \[
    \int_{x_0}^{x_1} [y(x) - I_hy(x)]^2 \, dx \leq \frac{1}{4} \, h^4 \,
    \int_{x_0}^{x_1} [y''(x)]^2 \, dx .
  \]
  For the last interval $(x_{n-1},x_n)$ the proof follows the same lines.
  When summing up the local contributions, the error estimates
  \eqref{Appendix Interpolation Error L2 H2} and
  \eqref{Appendix Interpolation Error H1 H2} follow.
  Moreover, for $i=2,\ldots,n-1$ and $x \in (x_{i-1},x_i)$, we write
  \[
    (I_hy)'(x) = \frac{1}{h} [y(x_i)-y(x_{i-1})] =
    \frac{1}{h} \int_{x_{i-1}}^{x_i} y'(s) \, ds,
  \]
  i.e.,
  \[
    [(I_hy)'(x)]^2 =
    \left( \frac{1}{h} \int_{x_{i-1}}^{x_i} y'(s) \, ds \right)^2 \leq
    \frac{1}{h} \int_{x_{i-1}}^{x_i} [y'(s)]^2 \, ds .
  \]
  Integration and summing up over all $i=2,\ldots,n-1$ gives, recall
  $(I_hy)'(x)=0$ for $x \in [x_0,x_1] \cup [x_{n-1},x_n]$,
  \[
    \| (I_hy)' \|^2_{L^2(0,1)} = \sum\limits_{i=2}^{n-1}
    \| (I_hy)' \|^2_{L^2(x_{i-1},x_i)} \leq \sum\limits_{i=2}^{n-1}
    \| y' \|^2_{L^2(x_{i-1},x_i)} \leq \| y' \|^2_{L^2(0,1)} ,
  \]
  i.e., \eqref{Appendix Interpolation Stability}.
  With this we finally conclude \eqref{Appendix Interpolation Error L2 H1},
  \[
    \| y - I_hy \|_{L^2(0,1)} \leq \frac{1}{\sqrt{2}} \, h \,
    \| y' - (I_hy)' \|_{L^2(0,1)} \leq
    \sqrt{2} \, h \, \| y' \|_{L^2(0,1)} .
  \]
\end{proof}

\noindent
Next we introduce the $L^2$ projection $Q_h : L^2(0,1) \to
\widetilde{S}^1_h(0,1)$, and present related error and stability estimates.

\begin{lemma}
  For any given $y \in L^2(0,1)$ we define
  $Q_hy \in \widetilde{S}_h^1(0,1)$ as unique
  solution of the variational formulation
  \begin{equation}\label{Appendix Def Qh}
    \langle Q_h y , y_h \rangle_{L^2(0,1)} =
    \langle y , y_h \rangle_{L^2(0,1)} \quad \mbox{for all} \;
    y_h \in \widetilde{S}_h^1(0,1) . 
  \end{equation}
  Then,
  \begin{equation}\label{Appendix L2 Stability Qh}
    \| Q_h y \|_{L^2(0,1)} \leq \| y \|_{L^2(0,1)} \quad \mbox{for all} \;
    y \in L^2(0,1).
  \end{equation}
  For $y \in H^2(0,1)$ with $y'(0)=y'(1)=0$ there hold the error estimates
  \begin{equation}\label{Appendix Projektion Error L2 H2}
    \| y - Q_h y \|_{L^2(0,1)} \leq \frac{1}{2} \, h^2 \,
    \| y'' \|_{L^2(0,1)},
  \end{equation}
  and
  \begin{equation}\label{Appendix Projektion Error H1 H2}
    \| y' - (Q_hy)' \|_{L^2(0,1)} \leq c \, h \,
    \| y'' \|_{L^2(0,1)} .
  \end{equation}
  Moreover, there holds the stability estimate
  \begin{equation}\label{Appendix Projektion H1 Stability}
    \| (Q_h y)' \|_{L^2(0,1)} \leq c \, \| y' \|_{L^2(0,1)} \quad
    \mbox{for all} \; y \in H^1(0,1), y'(0)=y'(1)=0.
  \end{equation}
\end{lemma}

\begin{proof}
  The stability estimate \eqref{Appendix L2 Stability Qh} is a consequence
  of the variational formulation \eqref{Appendix Def Qh}
  when choosing $y_h = Q_hy$. By Cea's lemma and using
  \eqref{Appendix Interpolation Error L2 H2} we have
  \eqref{Appendix Projektion Error L2 H2},
  \[
    \| y - Q_h y \|_{L^2(0,1)} \leq
    \inf\limits_{y_h \in Y_h} \| y - y_h \|_{L^2(0,1)} \leq
    \| y - I_h y \|_{L^2(0,1)} \leq \frac{1}{2} \, h^2 \,
    \| y'' \|_{L^2(0,1)} .
  \]
  Moreover, using the triangle inequatlity twice, the inverse inequality
  \[
    \| y_h' \|_{L^2(0,1)} \leq \sqrt{12} \, h^{-1} \, \| y_h \|_{L^2(0,1)}
    \quad \mbox{for} \; y_h \in \widetilde{S}_h^1(0,1),
  \]
  and the error extimates \eqref{Appendix Interpolation Error H1 H2} as well as
  \eqref{Appendix Interpolation Error L2 H2} and
  \eqref{Appendix Projektion Error L2 H2}, this gives
  \eqref{Appendix Projektion Error H1 H2},
  \begin{eqnarray*}
    \| y' - (Q_hy)' \|_{L^2(0,1)}
    & \leq & \| y' - (I_hy)' \|_{L^2(0,1)} +
             \| (I_hy)'-(Q_hy)' \|_{L^2(0,1)} \\
    && \hspace*{-2.5cm}
       \leq \, \frac{1}{\sqrt{2}} \, h \, \| y'' \|_{L^2(0,1)} +
       2\sqrt{3} \, h^{-1} \, \| I_h y - Q_h y \|_{L^2(0,1)} \\
    && \hspace*{-2.5cm}
       \leq \, \frac{1}{\sqrt{2}} \, h \, \| y'' \|_{L^2(0,1)} +
       2\sqrt{3} \, h^{-1} \, \Big[
       \| I_h y - y \|_{L^2(0,1)} + \| y - Q_h y \|_{L^2(0,1)}
       \Big] \\
    && \hspace*{-2.5cm}
       \leq \, \frac{1}{2} \Big( \sqrt{2}+ 4 \sqrt{3} \Big)
       \, h \, \| y'' \|_{L^2(0,1)} .
  \end{eqnarray*}
  When using the same arguments we finally have
  \begin{eqnarray*}
    \| (Q_hy)' \|_{L^2(0,1)}
    & \leq & \| (Q_hy)' - (I_hy)'\|_{L^2(0,1)} + \| (I_hy)' \|_{L^2(0,1)} \\
    & \leq & 2 \sqrt{3} \, h^{-1} \| Q_hy-I_hy \|_{L^2(0,1)} +
             \| y' \|_{L^2(0,1)} \\
    & \leq & 2 \sqrt{3} \, h^{-1} \Big[
             \| Q_hy- y \|_{L^2(0,1)} + \| y - I_hy \|_{L^2(0,1)} \Big]
             + \| y' \|_{L^2(0,1)} \\
    & \leq & (1 + 4 \sqrt{6}) \, \| y' \|_{L^2(0,1)} . 
  \end{eqnarray*}
\end{proof}

\noindent
For $\Omega = (0,1)^d$ we use the tensor product finite element space
$Y_h = \otimes_{i=1}^d \widetilde{S}_h^1(0,1)$, and we define
$P_h := \otimes_{i=1}^d Q_{h_i} : Y \to Y_h$, where 
$Q_{h_i} : L^2(0,1) \to \widetilde{S}_h^1(0,1)$ are defined as in
\eqref{Appendix Def Qh}, but with respect to the component
$x_i \in {\mathbb{R}}$ of $x \in {\mathbb{R}}^d$. 

\begin{lemma}
  For $y \in Y \cap H^2(\Omega)$ and $P_h y \in Y_h$, there hold
  the error estimates
  \begin{equation}\label{Appendix Error L2 H2}
    \| y - P_h y \|_{L^2(\Omega)} \leq c \, h^2 \, |y|_{H^2(\Omega)},
  \end{equation}
  and
  \begin{equation}\label{Appendix Error H1 H2}
    \| \nabla (y - P_h y) \|_{L^2(\Omega)} \leq c \, h \, |y|_{H^2(\Omega)}.
  \end{equation}
\end{lemma}

\begin{proof}
  For simplity, we consider the case $d=2$ only. Then, we conclude
  \begin{eqnarray*}
    \| y - P_h y \|_{L^2(\Omega)}
    & = & \| y - Q_{h_1} Q_{h_2} y \|_{L^2(\Omega)} \\
    & \leq & \| y - Q_{h_1} y \|_{L^2(\Omega)} +
             \| Q_{h_1} (y-Q_{h_2}y) \|_{L^2(\Omega)} \\
    & \leq & \| y - Q_{h_1} y \|_{L^2(\Omega)} + c \,
             \| y-Q_{h_2}y \|_{L^2(\Omega)} \\
    & \leq & c_1 \, h_1^2 \, \| \partial^2_{y_1y_1} y \|_{L^2(\Omega)} +
             c_2 \, h_2^2 \, \| \partial^2_{y_2y_2} y \|_{L^2(\Omega)} \\
    & \leq & c \, h^2 \, |y|_{H^2(\Omega)},
  \end{eqnarray*}
  when using the stability estimate \eqref{Appendix L2 Stability Qh},
  and the error estimate \eqref{Appendix Projektion Error L2 H2}, i.e.,
  \eqref{Appendix Error L2 H2}. The proof of
  \eqref{Appendix Error H1 H2} follows the same lines, but using the
  stability estimate \eqref{Appendix Projektion H1 Stability}, and the
  error estimate \eqref{Appendix Projektion Error H1 H2}.
\end{proof}

\noindent
It remains to prove a best approximation result for the boundary term as
used to derive the error estimate \eqref{Eqn:Error H2}.

\begin{lemma}
  Let $\Gamma = \partial \Omega$ for $\Omega = (0,1)^d$.
  Assume $y \in H^2(\Gamma)$, i.e., there exists an extension
  $y_e \in Y \cap H^{5/2}(\Omega)$. 
  Then, there holds the error estimate
  \begin{equation}
    \| y - P_h y \|_{L^2(\Gamma)} \leq c \, h^2 \, |y|_{H^2(\Gamma)} .
  \end{equation}
\end{lemma}

\begin{proof}
  Again, we consider the case $d=2$ only.
  Let $\Gamma_0 := \{ (x,0) : x \in (0,1) \}$. With
  \eqref{Appendix Interpolation Error L2 H2} and
  \eqref{Appendix L2 Stability Qh} we then have
  \begin{eqnarray*}
    \| y - P_h y \|_{L^2(\Gamma_0)}
    & = & \| y - Q_{h_1} Q_{h_2}y \|_{L^2(\Gamma_0)} \\
    & \leq & \| y - Q_{h_1} y \|_{L^2(\Gamma_0)} +
             \| Q_{h_1} (y-Q_{h_2}y) \|_{L^2(\Gamma_0)} \\
    & \leq & \frac{1}{2} \, h^2 \, \| \partial_{x_1x_1} y \|_{L^2(\Gamma_0)} +
             \| y - Q_{h_2} y \|_{L^2(\Gamma_0)} \\
    & \leq & \frac{1}{2} \, h^2 \, \| y \|_{H^{5/2}(\Omega)} +
             \| y - I_{h_2} y \|_{L^2(\Gamma_0)} .
  \end{eqnarray*}
  For $x \in \Gamma_0$ we can write, using $\partial_{x_2}y_e(x,s)_{|s=0}=0$,
  \begin{eqnarray*}
    y(x,0) - I_{h_2}y(x,0)
    & = & y_e(x,0) - y_e(x,h) = - \int_0^h \partial_{x_2} y_e(x,s) \, ds \\
    && \hspace*{-2cm}
       = \, \int_0^h [\partial_{x_2}y_e(x,0) - \partial_{x_2}y_e(x,s)] \, ds
       \, = \, - \int_0^h \int_0^s \partial_{x_2x_2}y_e(x,\tau) \, d\tau ds .
  \end{eqnarray*}
  Hence we obtain
   \begin{eqnarray*}
     [y(x,0) - I_{h_2}y(x,0)]^2
     & = &  \left[
           \int_0^h \int_0^s 1 \cdot
           \partial_{x_2x_2}y_e(x,\tau) \, d\tau ds
           \right]^2 \\
     & \leq & \int_0^h 1^2 \, ds \int_0^h \left[
              \int_0^s 1 \cdot
              \partial_{x_2x_2}y_e(x,\tau) \, d\tau \right]^2 ds \\
     & \leq & h \int_0^h \int_0^s 1^2 \, d\tau
              \int_0^s [\partial_{x_2x_2}y_e(x,\tau)]^2 d\tau ds \\
     & \leq & h \int_0^h s \, ds \int_0^h [\partial_{x_2x_2}y_e(x,\tau)]^2
              d\tau \\
     & = & \frac{1}{2} \, h^3 \,
           \int_0^h [\partial_{x_2x_2}y_e(x,\tau)]^2 d\tau .
   \end{eqnarray*}
   Next we consider $x \in (x_{i-1},x_i)$, $i=1,\ldots,n$, and
   use the piecewise constant $L^2$ projection
   \[
     \hat{y}_i(\tau) = \frac{1}{h} \int_{x_{i-1}}^{x_i}
     \partial_{x_2x_2 }y_e(\eta,\tau) \, d\eta,
   \]
   to write
   \[
     [\partial_{x_2x_2}y_e(x,\tau)]^2
     \leq
     2 \, [\partial_{x_2x_2}y_e(x,\tau) - \hat{y}_i(\tau)]^2 +
     2 \, [\hat{y}_i(\tau)]^2,
   \]
   i.e.,
   \[
     [y(x,0) - I_{h_2}y(x,0)]^2
     \, \leq \, h^3 \int_0^h
     [\partial_{x_2x_2}y_e(x,\tau) - \hat{y}_i(\tau)]^2 d\tau
     + h^3 \int_0^h [\hat{y}_i(\tau)]^2 d\tau ,
   \]
   and integration over $x \in (x_{i-1},x_i)$ gives
   \begin{eqnarray*}
     && \int_{x_{i-1}}^{x_i} [y(x,0) - I_{h_2}y(x,0)]^2 dx \\
     && \hspace*{1cm}\leq \, h^3 \int_{x_{i-1}}^{x_i }\int_0^h 
        [\partial_{x_2x_2}y_e(x,\tau) - \hat{y}_i(\tau)]^2 d\tau dx
        + h^3 \int_{x_{i-1}}^{x_i} \int_0^h [\hat{y}_i(\tau)]^2 d\tau dx .
   \end{eqnarray*}
   Using
   \begin{eqnarray*}
     \int_{x_{i-1}}^{x_i}
     [\partial_{x_2x_2}y_e(x,\tau) - \hat{y}_i(\tau)]^2 dx
     & = & \int_{x_{i-1}}^{x_i}
           \left[
           \partial_{x_2x_2} y_e(x,\tau) - \frac{1}{h} \int_{x_{i-1}}^{x_i}
           \partial_{x_2x_2} y_e(\eta,\tau) \, d\eta
           \right]^2 dx \\
     && \hspace*{-3cm} = \, \int_{x_{i-1}}^{x_i} \left[
           \frac{1}{h} \int_{x_{i-1}}^{x_i}
           [\partial_{x_2x_2}y_e(x,\tau)-\partial_{x_2x_2}y_e(\eta,\tau)] \, d\eta
           \right]^2 dx \\
     && \hspace*{-3cm} = \, \frac{1}{h^2} \int_{x_{i-1}}^{x_i} \left[
           \int_{x_{i-1}}^{x_i}
           \frac{\partial_{x_2x_2}y_e(x,\tau)-\partial_{x_2x_2}y_e(\eta,\tau)}
           {|x-\eta|} \, |x-\eta| \, d\eta
           \right]^2 dx \\
     && \hspace*{-3cm} \leq \, \frac{1}{h^2} \int_{x_{i-1}}^{x_i}
        \int_{x_{i-1}}^{x_i} |x-\eta|^2 \, d\eta
        \int_{x_{i-1}}^{x_i}
        \frac{[\partial_{x_2x_2}y_e(x,\tau)-\partial_{x_2x_2}y_e(\eta,\tau)]^2}
           {|x-\eta|^2} \, d\eta dx \\
     && \hspace*{-3cm} \leq \, h \, \int_{x_{i-1}}^{x_i}
        \int_{x_{i-1}}^{x_i}
        \frac{[\partial_{x_2x_2}y_e(x,\tau)-\partial_{x_2x_2}y_e(\eta,\tau)]^2}
           {|x-\eta|^2} \, d\eta dx,
   \end{eqnarray*}
   we conclude
   \begin{eqnarray*}
     && \int_{x_{i-1}}^{x_i} [y(x,0) - I_{h_2}y(x,0)]^2 dx \\
     && \hspace*{1cm}\leq \, h^4 \, \int_{x_{i-1}}^{x_i} \int_{x_{i-1}}^{x_i}
        \frac{[\partial_{x_2x_2}y_e(x,\tau)-\partial_{x_2x_2}y_e(\eta,\tau)]^2}
        {|x-\eta|^2} \, d\eta dx
        + h^4 \int_0^h [\hat{y}_i(\tau)]^2 d\tau  \\
     && \hspace*{1cm} \leq \, h^4 \, |y_e|^2_{H^{5/2}((x_{i-1},x_i)\times(0,h))}.
   \end{eqnarray*}
   When summing up all contrubutions, this finally gives
   \[
     \| y - P_hy \|_{L^2(\Gamma_0)} \leq
     c \, h^2 \, |y_e|_{H^{5/2}(\Omega)}.
   \]
   For all other boundary edges, the proof follows the same lines.
\end{proof}

\end{document}